\documentclass[11 pt,reqno]{amsart}
\usepackage{amsmath}
\usepackage{amsfonts}
\usepackage{amssymb}
\usepackage{amsthm}
\usepackage{amscd}
\usepackage{a4wide}
\usepackage{enumerate}
\usepackage{pifont}
\usepackage{dsfont} 
\usepackage{xcolor}
\usepackage{amsmath}
\usepackage{tikz}
\usepackage{tikz-cd}
\usepackage{booktabs}
\usetikzlibrary{shapes.geometric}
\usepackage{amscd}
\usepackage{mathabx}
\numberwithin{equation}{section}
\usepackage[original]{imakeidx}
\makeindex 
\usepackage{mathrsfs}
\usepackage{graphics}
\usepackage{eucal}
\usepackage{mathrsfs}
\usepackage{mdwlist}
\usepackage{color}
\usepackage[pdftex,bookmarks,colorlinks,breaklinks]{hyperref}  
\usepackage{mdwlist}
\usepackage{environ}
\hypersetup{linkcolor=red,citecolor=blue,filecolor=dullmagenta,urlcolor=darkblue} 
\usepackage{enumitem}
\usepackage{multicol}
\setenumerate[1]{label=\thesection.\arabic*.}
\newif\ifhinting\hintingtrue 
\usepackage{fancyhdr}
\fancypagestyle{plain}{ %
  \fancyhf{} 
  
}
\newtheorem {theorem}    {Theorem}[section]
\newtheorem {lemma}      [theorem]    {Lemma}

\newtheorem {corollary}  [theorem]    {Corollary}
\newtheorem {proposition}[theorem]    {Proposition}

\theoremstyle{definition}
\newtheorem {definition} [theorem]    {Definition}

\newtheorem {remark}    [theorem]    {Remark}

\newcounter{AbcT}

\numberwithin{equation}{section}

\newcommand {\N} {{\mathbb N}}
\newcommand {\R} {{\mathbb R}}

\newcommand{\Q}{{\mathbb Q}}

\newcommand {\Z} {{\mathbb Z}}

\newcommand{\IGNORE}[1]{}

\renewcommand{\liminf}{\varliminf}

 \DeclareMathOperator{\SL}{SL}

 \DeclareMathOperator{\SO}{SO}

\newcommand\Lie{\operatorname{Lie}}

\newcommand{\ve}{\varepsilon}

\newcommand{\bs}{\boldsymbol}

\begin{document}
\title{A survey and a result on inhomogeneous quadratic forms}

\begin{abstract}
  
  We survey recent work done on the values at integer points of irrational inhomogeneous quadratic forms, namely, inhomogeneous analogues of the famous Oppenheim conjecture. We also prove that the set of such forms in two variables whose set of values at integer points avoids a given countable set not containing zero, has full Hausdorff dimension. Moreover, we consider the more refined variant of this problem, where the shift is fixed and the form is allowed to vary. The strategy is to translate the problems to homogeneous dynamics and deduce the theorems from their dynamical counterparts. While our approach is inspired by the work of Kleinbock and Weiss \cite{KW}, the dynamical results can be deduced from more general results of An, Guan, and Kleinbock \cite{AGK}.   
 
\end{abstract}

\thanks{A.\ G.\ gratefully acknowledges support from a grant from the Infosys foundation to the Infosys Chandrasekharan Random Geometry Centre and a J. C. Bose grant from the ANRF. S. \ D.\ and  A.\ G.\ gratefully acknowledge a grant from the Department of Atomic Energy, Government of India, under project $12-R\&D-TFR-5.01-0500$.}
 

\author{ Sourav Das}

\author{ Anish Ghosh}
 
\address{School of Mathematics, Tata Institute of Fundamental Research, Mumbai, India 400005 }
\email{ iamsouravdas1@gamil.com/sourav@math.tifr.res.in}

\address{School of Mathematics, Tata Institute of Fundamental Research, Mumbai, India 400005 }

\email{ ghosh@math.tifr.res.in}

\maketitle
\section{Introduction}\label{section:intro}
An inhomogeneous quadratic form is a quadratic form accompanied by a shift. Given a quadratic form $Q$ on $\mathbb R^n$ and a vector $\xi \in \mathbb R^n$, the inhomogeneous quadratic form $(Q,\xi)$ (also denoted by $Q_{\xi}$) is defined by
  \begin{equation*}
      Q_{\xi}( y):=Q(y+\xi)\,\, \text{for all}\,\, y\in \mathbb R^n.
  \end{equation*} 

The values at integer points of such quadratic forms have been intensively studied. In addition to their inherent arithmetic qualities, they connect to a wide variety of mathematical subjects. The last ten years or so have especially seen a flurry of activity fueled by the introduction of a variety of new ideas. In this article, we will first survey some of this activity and then prove some results concerning binary inhomogeneous forms.\\

Let us begin with a more classical situation, namely the ``homogeneous" case where $\xi = 0$. If we assume that $Q$ is an indefinite rational quadratic form, then a classical theorem of Meyer states that as long as $Q$ has at least five variables, the equation $Q(y) = 0$ has a non-trivial integer solution. Meyer's theorem is a consequence of the Hasse-Minkowski local-global principle.\\

The British mathematician Oppenheim wondered what the situation would be if one were to consider a similar question for \emph{irrational} quadratic forms. One would certainly not expect Meyer's theorem to hold, but Oppenheim conjectured that it would almost hold. Namely, if $Q$ is an indefinite, irrational quadratic form in at least $3$ variables\footnote{Initially, Oppenheim stated his conjecture for $\geq 5$ variables, in harmony with Meyer's result, but he later upgraded his conjecture to $3$ or more variables.} then for every real number $\alpha$ and for every $\varepsilon > 0$, one can solve the \emph{inequality} $|Q(y) - \alpha| < \varepsilon$ in integer triples $y$. After several developments using the circle method, an important observation of Raghunathan converting Oppenheim's conjecture into a question in homogeneous dynamics enabled Margulis to resolve this conjecture. Margulis's work is a landmark in homogeneous dynamics.\\

It is well known that  $\SL(n, \mathbb{R})/\SL(n, \mathbb{Z})$ can be identified with the space of unimodular lattices in $\mathbb{R}^n$. A subgroup of $\SL(n, \mathbb{R})$ acts on $\SL(n, \mathbb{R})/\SL(n, \mathbb{Z})$ by translation and the resulting dynamics is rich and complicated. It is easy to see that Oppenheim's conjecture can be reduced to $3$ variables. Raghunathan's key observation was the following: Oppenheim's conjecture follows from the statement that any orbit of $\SO(Q)$ (the isotropy group of the quadratic form in $3$ variables, considered a subgroup of $\SL(3, \R)$) on $\SL(3, \mathbb{R})/\SL(3, \mathbb{Z})$  is either closed and carries an $\SO(Q)$-invariant probability measure, or is dense.

  The orthogonal group of such a form is a conjugate of $\SO(2, 1)$, and is generated by unipotent one-parameter subgroups. Inspired by seminal work of Dani, Furstenberg and Veech on the measure and topological rigidity of horocycle flows on hyperbolic surfaces;  Raghunathan conjectured very general topological rigidity statements for such actions, and Dani conjectured very general measure rigidity statements.
  
  Margulis proved an instance of Raghunathan's conjecture for the isotropy group $\SO(Q)$ acting on $\SL(3, \mathbb{R})/\SL(3, \mathbb{Z})$, thereby proving Oppenheim's conjecture.  In full generality, the conjectures of Raghunathan and Dani were proved in a series of landmark papers by Marina Ratner. Returning to quadratic forms, we refer the reader to Borel's survey \cite{B} on Oppenheim's conjecture for a very readable account of the early history of this beautiful story.\\

Modern work on the theme of this article, namely \emph{inhomogeneous forms}, begins with the important work of Marklof \cite{Marklof1, Marklof2} on pair correlation densities of inhomogeneous quadratic forms. This is intimately connected to the statistical behaviour of energy levels of quantum systems. In \cite{Marklof1}, Marklof partly settles a conjecture of Berry and Tabor. More precisely he shows that the 
local two-point correlations of the sequence given by the values $(m - \alpha)^2 + (n - \beta)^2, \text{ with } (m, n) \in \Z^2$, are those of a Poisson process under explicit diophantine conditions on $(\alpha, \beta) \in \R^2$. The methods used are an ingenious combination of theta sums and the Shale-Weil representation and Ratner's theorems on unipotent flows mentioned above. In \cite{Marklof2}, this study is extended to higher dimensions.\\

Henceforth, we will further assume that $Q$ is indefinite and that the pair $(Q, \xi)$ is \emph{irrational}, namely, either $Q$ is not proportional to a quadratic form with integer coefficients, or the vector $\xi$ is an irrational vector, i.e. not proportional to an integer vector. If $Q$ has more than $3$ variables, then an inhomogeneous version of Oppenheim's conjecture can be seen to hold. This is implied from a stronger, quantitative statement in \cite{MM}; alternatively, the reader can read a self contained proof in \cite{BG}, which also contains density results for systems of forms. In other words, for an irrational form $(Q, \xi)$ in at least $3$ variables, and for every real number $\alpha$, the inequality 
\begin{equation}\label{opp:inhom}
|Q_{\xi}(y)- \alpha| < \varepsilon
\end{equation}
can always be solved in integer vectors $y$.\\

 Margulis and Mohammadi \cite{MM} proved a quantitative version of Oppenheim's conjecture for inhomogeneous forms of the above type. To set the stage for their results, we introduce the counting function

\begin{equation}
N_{Q_{\xi},I}(t)=\#\{y \in \mathbb{Z}^n\ |\  Q_{\xi}(y)\in I,\ \|y\|\leq t\},
\end{equation}
where $\|\cdot\|$ is the Euclidean norm on $\mathbb R^n$ and $I\subseteq \mathbb R$ is an interval. In loc. cit. the following theorem is proved. 
\begin{theorem}
For any indefinite, irrational and nondegenerate inhomogeneous quadratic form $Q_{\xi}$ in $n \geq 3$ variables, there is $c_Q>0$ such that 
$$\liminf_{t\to \infty}\frac{N_{Q_{\xi},I}(t)}{t^{n-2}}\geq c_Q |I|,$$
while for $n\geq 5$ the limit exists and equals $c_Q|I|$.
\end{theorem}
The homogeneous counterpart was proved earlier in three landmark works. First,  Dani and Margulis proved a lower bound in \cite{DM1} and subsequently, Eskin, Margulis, and Mozes proved an upper bound in the two papers \cite{EMM1, EMM2}. The proof of Margulis and Mohammadi is a very non-trivial adaptation of these approaches to inhomogeneous forms. We will be concerned with more refined questions in this survey: namely, we will pursue error terms in the count, as well as ``effective" bounds on the size of the integer vector which solves the inequality (\ref{opp:inhom}). We will explain these in greater detail below.

\section{Questions of effectivity}

What is effectivity? We quote the formulation in \cite{GKY1}: for a given $Q_{\xi}, \alpha \in \R$ and $t \geq 1$ large:
\begin{itemize}
\item How small can $|Q_{\xi}(y)- \alpha|$ get for $y \in \Z^n$ with $\|y\| \leq t$ bounded?\\
\item Is it possible to obtain an eﬀective estimate (i.e., with a remainder)
for the counting function $N_{Q_{\xi}, I}(t)$?
\end{itemize}
During the last few years, significant progress has been made on these questions for ternary homogeneous forms; beginning with the important work of Lindenstrauss and Margulis \cite{LinMar}, and the recent breakthrough \cite{LMWY}.  We also refer the reader to the lovely argument in \cite{Kelmer} where effective bounds are obtained for indefinite ternary inhomogeneous quadratic forms. The results quoted above use sophisticated techniques and provide ample indication that effective results are very challenging to prove. They are also quite precise in nature, namely, they provide an explicit Diophantine condition to be satisfied by a ternary homogeneous quadratic form $Q$, under which effective results can be obtained. We note that the situation is also favourable if the number of variables is large because the circle method and its modern variants offer very good effective results, see also \cite{BGHM, SV}. However, Oppenheim type problems can be phrased in far greater generality, and so the effective Oppenheim problem can be said to be generally wide open.\\

In such a situation, it is natural to turn a probabilistic lens on Oppenheim's conjecture. That is, the space of nondegenerate indefinite ternary quadratic forms in $3$ variables of determinant $1$ can be viewed as the homogeneous space $\SO(2,1)\backslash \SL(3, \R)$, and thus has a natural measure on it. One can ask, what proportion of quadratic forms satisfies a version of the quantitative Oppenheim conjecture with a good error rate, or what proportion of quadratic forms satisfies an ``effective" version of Oppenheim's conjecture, especially an optimally effective version? This turns out to be a remarkably robust way of stating the problem. This style of thinking can already be glimpsed in Theorem 2.4 of \cite{EMM1}, which proves an almost sure quantitative result for quadratic forms of signature $(2,1)$ or $(2,2)$ with error terms using an ergodic theorem of Nevo and Stein, and an earlier article of Sarnak \cite{Sa}.\\

Many different density problems can be framed in this manner and over the last decade two distinct strategies have evolved to study the `probabilistic' Oppenheim conjecture in it's many different forms.\\

One approach involves the use of higher moments of Siegel transforms, especially Rogers's seminal work. This was developed by Athreya and Margulis \cite{AM}, and followed by Kelmer and Yu \cite{KY}, \cite{GKY1}, \cite{KS} amongst others. We will discuss this approach through the results in the paper \cite{GKY1} of Ghosh, Kelmer and Yu which studies inhomogeneous forms, the main topic of this paper.\\

The second approach is completely different and involves performing Diophantine analysis using effective ergodic theorems for actions of semisimple groups. This approach was developed by Ghosh, Gorodnik and Nevo in the context of Diophantine approximation on homogeneous varieties of semisimple groups \cite{GGN1, GGN2, GGN3, GGN4, GGN6} and has been applied to a wide variety of Oppenheim type problems involving polynomials in \cite{GGN5} and by Ghosh and Kelmer to quadratic forms in \cite{GKe0, GKe}. We will illustrate this approach using the work \cite{GKY2} of Ghosh, Kelmer and Yu, which again deals with inhomogeneous quadratic forms.\\

Before proceeding, we should frame the effectivity problem for inhomogeneous forms. Due to the presence of two variables, i.e. quadratic form and shift, the problem can be framed in $3$ distinct ways. That is, one could consider the following options.

\begin{itemize}
\item vary $Q$ as well as $\xi$,\\
\item fix $\xi$ and vary $Q$,\\
\item fix $Q$ and vary $\xi$.
\end{itemize}

It turns out that the answer to the first possibility, namely, when one is allowed to vary both parameters, is relatively simple. In other words, one can apply in a straightforward fashion, a version of Rogers's second moment formula on the space of affine lattices (see \cite[Lemma 4]{A} and \cite[formula 3.7]{EMV}) to get the following theorem. 


\begin{theorem}\label{thm:1}
Let $n \geq 3$ and let $0 \leq \kappa < n - 2$. Let $\{I_t\}_{t>0}$ be a decreasing family of
bounded measurable subsets of $\mathbb R$ with measures $|I_t|= ct^{-\kappa}$ for some $c > 0$. Then there is
$\nu > 0$ such that for almost every non-degenerate indefinite quadratic form $Q$ in $n$ variables
there is a constant $c_Q$ such that for almost every $\xi \in \R^n$, we have
$$N_{Q_{\xi},I_t}(t) = c_Q|I_t|t^{n-2} + O_{Q,\xi}(t^{n-2-\kappa-\nu} ).$$
In particular, for any $\kappa < n - 2$, for almost every $\xi \in \R^n$, and for almost every quadratic
form $Q$ as above, the system of inequalities
$$ |Q_{\xi}(y)- \alpha|< t^{-\kappa}, \|y\| \leq t$$
has integer solutions for all suﬃciently large $t$.

\end{theorem}

The epithet ``simple" should be taken with a large grain of salt. The development and use of an affine version of Rogers's formula is by no means easy. Additionally, one needs to carefully estimate the volume of certain regions, which, again, is far from routine. What is meant is that Theorem \ref{thm:1} more or less follows, once the strategy to prove the corresponding result for homogeneous forms is known.\\  

The situation in which one of the parameters is fixed does not fit into this mould and turns out to be much more delicate. 

\subsection{Fixed rational shifts}
We now fix $\xi\in \mathbb{Q}^n$, and allow ourselves to vary $Q$. In order to study this, the following observation was key to the strategy employed by Ghosh, Kelmer and Yu. If $\xi = \frac{\mathbf{p}}
{q}$ is a rational shift, then the values of the
shifted form at integer points
$$ Q_{\xi}(y) = Q(y + \frac{\mathbf{p}}{q} ) = q^{-2}Q(qy + \mathbf{p}),$$
are just a scaling by $\frac{1}{q^2}$ of the values of the homogeneous form $Q$ evaluated on integer points that satisfy a congruence condition modulo $q$. The main technical innovation in loc. cit. is the calculation of the second moment of a Siegel transform on a congruence quotient of the special linear group. The latter has already yielded other applications (cf. \cite{AGH, AGY}). This formula yields the following counting formula and effective result for inhomogeneous forms with fixed rational shift.

\begin{theorem}\label{thm:2}
Let $n\geq 3$ and let $0\leq\kappa<n-2$. Let $\xi\in \mathbb{Q}^n$ be a fixed rational vector. Let $\{I_t\}_{t> 0}$ be a decreasing family of bounded measurable subsets of $\mathbb{R}$ with measures $|I_t|=ct^{-\kappa}$ for some $c>0$. Then there is $\nu>0$ such that for almost every non-degenerate indefinite quadratic form $Q$ in $n$ variables, there is a constant $c_Q>0$ such that 
\begin{equation}\label{equ:countingresult}
N_{Q_{\xi,I_t}}(t)=c_Q |I_t|t^{n-2}+O_{Q,\xi}(t^{n-2-\kappa-\nu}).
\end{equation}
In particular, for any $\kappa<n-2,$ for every $\xi\in\mathbb{Q}^n,$ and for almost every quadratic form $Q$ as above, the system of inequalities 
\begin{equation} \label{equ:maininequ} 
|Q_{\xi}(y)-\alpha|<t^{-\kappa},\; \|y\|\leq t
\end{equation}
has integer solutions for all sufficiently large $t$.
\end{theorem}

Moreover, Theorem \ref{thm:2} yields the following Corollary.

\begin{corollary}
Let $n \geq 3$ and let $0 \leq \kappa < n- 2$ and let $q \in \N$. Let $\{I_t\}_{t>0}$ be a decreasing
family of bounded measurable subsets of $\R$ with measures $|I_t|= ct^{-\kappa}$ for some $c > 0$. Then there is $\nu > 0$ such that for almost every non-degenerate indefinite quadratic form $Q$ in $n$
variables, there is a constant $c_Q > 0$ such that for all $\mathbf{p} \in (\Z/q\Z)^n$,
$$\#\{w \in \Z^n~|~ Q(w) \in I_t, w \equiv \mathbf{p} ~(\text{ mod } q), \|w\|< t\} = c_Q|I_t|t^{n-2} + O_{Q,\nu,q}(t^{n-2-\kappa-\nu}).$$
\end{corollary}

The situation of a fixed irrational shift is not very satisfactory. Currently, the following theorem, also proved in \cite{GKY1} is state of the art but is sub-optimal and depends on approximating the fixed irrational shift using rational shifts. 

\begin{theorem}\label{thm:3}
Let $n \geq 5$ and let $\xi \in \R^n$ be an irrational vector with Diophantine exponent $\omega_{\xi}$ and uniform Diophantine exponent $\tilde{\omega}_{\xi}$. Assume that $\omega_{\xi} < \infty$ and that $\tilde{\omega}_{\xi} >\frac{2}{n-2}$. Then for any $\kappa \in \left(0, \frac{(n-2)\tilde{\omega}_{\xi} - 2}{n(1 + \omega_{\xi} + \tilde{\omega}_{\xi}\omega_{\xi})} \right)$, for any $\alpha \in \R$ and for almost every
non-degenerate indefinite quadratic form $Q$ in $n$ variables, the system of inequalities (\ref{equ:maininequ}) has integer solutions for all sufficiently large $t$.
\end{theorem}

We refrain from defining the uniform Diophantine exponent here, referring the reader to section 5 of \cite{GKY2}. The condition $n \geq 5$ has been imposed in order to ensure that the set of irrational vectors satisfying the Diophantine conditions in Theorem \ref{thm:3} is non-empty. The exponent obtained in Theorem \ref{thm:3} is also far from optimal. The study of rational forms with fixed irrational shift remains a challenging open problem.\\ 

The problem of fixing a rational quadratic form and letting the shift vary is of a completely different flavour and consequently needs completely different methods to deal with. This is accomplished in \cite{GKY2} and the method of attack involves using the ergodic theory of semisimple groups, an approach earlier used in \cite{GGN5} and \cite{GKe0,GKe}.

\begin{theorem}\label{thm:4}
For any rational indefinite form $Q$ in $n$ variables and any $\alpha\in \mathbb{R},$
there is $\kappa_0$ (depending only on the signature of $Q$) such that for any $\kappa<\kappa_0,$  for almost all $\xi\in \mathbb{R}^n$ the system of inequalities $$|Q(y+\xi)-\alpha|<t^{-\kappa},\quad \|y\|<t$$
has integer solutions for all sufficiently large $t$.
\end{theorem} 
The proof of Theorem \ref{thm:4} gives the following explicit values for $\kappa_0$, depending only on the signature of  $Q$ 
$$\kappa_0= \left\lbrace\begin{array}{ll} 
1 & (p,q)=(2,1)\\2 & (p,q)=(n - 1,1),\,\, n\geq 4\\
2& (p,q)=(2,2)\\
3/2 & (p,q)=(4,2) \mbox{ or } (3,3)\\
5/2 & (p,q)=(6,3),
\end{array}\right.
$$
while for all other signatures $p\geq q>1$ with $p+q=n$ we have $\kappa_0=2\kappa_1q(p-1)$ with 
$$\kappa_1= \left\lbrace\begin{array}{ll} 
\tfrac{1}{n} & n \equiv0\pmod{4}\\
\frac{1}{n-1} & n\equiv 1\pmod{4}\\
\tfrac{1}{n-2}& n\equiv 2\pmod{4}\\
\tfrac{1}{n+1} & n\equiv 3\pmod{4}.
\end{array}\right.
$$

Theorem \ref{thm:4} gives the optimal rate for $n = 3$ and $4$, namely  $\kappa_0 = n - 2$. However, when $n \geq 5$ is not as good as what we know for generic forms. We now explain briefly the significance of $\kappa_0$ and how it appears in Theorem \ref{thm:4}. The starting point of the ``ergodic" approach to studying effective results for quadratic forms is the realization that the isotropy group of the form is a semisimple group, for which mean ergodic theorems with rates can be proved. Let $G = \SO^{+}_Q(\R)$ denote the connected component of the identity in the isotropy group of $Q$. If $Q$ is a rational form, then $\Gamma = \SO^{+}_Q(\Z)$ is a lattice in $G$. Using the natural embedding of $G$ in $\SL(n, \R)$, we get an action of $G$ on $\R^n$ and we may consider the semi-direct product $\tilde{G} =G \ltimes \R^n$. Note that $ \tilde{\Gamma} = \Gamma \ltimes \Z^n$ is a non-uniform lattice in $\tilde{G}$ and that there is a natural left action of $G$ on the space $L^2( \tilde{G}/\tilde{\Gamma})$ preserving the probability Haar measure $m_{\tilde{G}}$ on $\tilde{G}/\tilde{\Gamma}$.

For any $f\in L^2(\tilde{G}/\tilde{\Gamma})$ and growing measurable subsets $G_t\subseteq G$ consider the averaging operator
\begin{equation}
\beta_{G_t}f(x)=\frac{1}{m_G(G_t)}\int_{G_t}f(g^{-1}x)dm_G(g),
\end{equation}
where $m_G$ denotes the probability measure on $G/\Gamma$ coming from the Haar measure on $G$. An effective mean ergodic theorem of the following kind
$$\left\|\beta_{G_t}f-\int_{\tilde{G}/\tilde{\Gamma}}fdm_{\tilde{G}} \right\|_2 \leq C \frac{\|f\|_2}{m_G(G_t)^\kappa},$$
valid for all $f\in L^2(\tilde{G}/\tilde{\Gamma})$, is intimately related to the so-called ``shrinking target problem" for the action of $G$ on $\tilde{G}/\tilde{\Gamma}$. 
 In \cite{GKY2}, it is shown that Theorem \ref{thm:4} follows from the following result.

\begin{theorem}\label{thm:5}
Let $Q$ be an indefinite rational form of signature $(n - q,q)$ and let $\kappa_1 = \frac{\kappa_0}{2q(n-q-1)}$ with $\kappa_0$ as above. There is a family of growing norm balls $G_t \subseteq \SO^{+}_Q(\R)$ of measure $m_G(G_t) \gg t^{q(n-q-1)}$ such that for any $\kappa < \kappa_1$ and for any $f \in L^{2}_{00}(\tilde{G}/\tilde{\Gamma})$ (i.e., the set of all functions whose average over $\mathbb R^n/\mathbb Z^n$ is zero), we have that $$\|\beta_{G_t} f\|_{2}\ll_{\kappa} \frac{\|f\|_2}{m_G(G_t)\kappa} $$
where the implied constant depends only on $\kappa$.
\end{theorem}

The proof of Theorem \ref{thm:5} uses a general spectral transfer principle, giving explicit bounds on the exponent in the mean ergodic theorem in terms of the strong spectral gap of the corresponding representation. In particular, determination of the strong spectral gap is crucial for the explicit exponent in Theorem \ref{thm:4}. 

\section{Uniformity and Averaging along smaller families}
In this section, we will discuss some results of Ghosh and Kumaraswamy \cite{GK} on inhomogeneous forms inspired by the work of Bourgain \cite{Bou}, who rephrased Margulis's theorem resolving Oppenheim's conjecture as follows: 
there is a sequence $\delta(N)\to 0$ (depending on $Q$) such that for all sufficiently large $N$,
\begin{equation*}
\min_{0 < |x| < N}|Q(x)|\leq \delta(N).
\end{equation*}
One can then ask for upper bounds for $\delta(N)$. This is yet another way of quantifying Oppenheim's conjecture, and one can investigate probabilistic versions of this question as before.\\

Rather than averaging over all quadratic forms, Bourgain treated the more difficult problem of averaging over a smaller family. Let $$
Q(x) = x_1^2 -\alpha_2 x_2^2 -\alpha_3x_3^2,
$$
then for any fixed $\alpha_2 > 0$ and almost all $\alpha_3 \in [1/2,1],$ Bourgain~\cite{Bou} showed that $\delta(N) \ll N^{-2/5+\ve}$. In the same article, assuming the Lindel\"of hypothesis for the Riemann zeta function, Bourgain improved this to $\delta(N) \ll N^{-1+\ve}$. This result is optimal and is a ``uniform in targets" effective version of the Oppenheim conjecture. Ghosh and Kelmer \cite{GKe}, obtained an analogue of Bourgain's theorem for generic ternary (i.e. the average is over all ternary forms and not just diagonal ones) quadratic forms using effective ergodic theorems for semisimple groups. Bourgain's result has been generalised to diagonal ternary forms of higher degree by the work of Schindler \cite{Schindler20}.\\ 

In \cite{GK}, diagonal ternary inhomogeneous forms are considered and the following quantitative theorem is proved.  

\begin{theorem}\label{thmshift}
Let $k \geq 2$ be an integer and let $Q(x) = x_1^k -\alpha_2x_2^k - \alpha_3 x_3^k$ and $\ve > 0$. Let $\bs{\theta} = (\theta_1,\theta_2,\theta_3) \in \R^3$ be a fixed vector. Then for any fixed $\alpha_2 >0$ and almost every $\alpha_3 \in [1/2,1]$ the following statements hold:
\begin{enumerate}[label=({{\roman*}})]
\item Assume the exponent pair conjecture. Then we have
$$
\min_{\substack{x \in \Z^3 \\ |x| \sim N}}|Q(x+\bs{\theta)}| \ll_{\ve,\alpha_2,\alpha_3,\bs{\theta}} N^{k-3+\ve}.
$$
\item Unconditionally, we have
$$
\min_{\substack{x \in \Z^3 \\ |x| \sim N}}|Q(x+\bs{\theta)}| \ll_{\ve,\alpha_2,\alpha_3,\bs{\theta}} N^{k-12/5+\ve}.
$$
\end{enumerate}
\end{theorem}
We refer to \cite{GK} for the statement of the exponent pair conjecture and further references, only mentioning that it implies the truth of the Lindel\"of hypothesis for the Riemann zeta function $\zeta(s)$.\\

In \cite{GK}, sharp upper bounds are also proved for the following Diophantine inequality in four variables, which is in the tradition of the problem of simultaneous inhomogeneous Diophantine approximation on curves. 
\begin{theorem}\label{prop1}
Let $\theta_1,\theta_2,\alpha, \beta$ be fixed real numbers. Let $0 < \delta < 1$. Define
$\mathcal{N}(M,\alpha,\delta)$ to be the number of solutions $(m_1,m_2,m_3,m_4) \in \Z^4 \cap [M,2M]^4$ to the inequality 
\begin{equation}\label{eq:nmaddef}
|(m_1+\theta_1)^\alpha -(m_2+\theta_1)^\alpha + \beta(m_3+\theta_2)^\alpha -  \beta(m_4+\theta_2)^\alpha| \leq \delta M^{\alpha}.
\end{equation}
Suppose that $\alpha \neq 0,1$. Then for any $\ve > 0$ we have $$\mathcal{N}(M,\alpha,\delta) \ll_{\alpha,\beta,\ve,\theta_1,\theta_2} M^{2+\ve}+\delta M^{4+\ve}.$$
\end{theorem}

\section{Values of quadratic forms at \texorpdfstring{$S$}{S} integers}
Let $S$ be a finite set of places of $\Q$ including the infinite place. Set $\Q_S := \prod_{s \in S}\Q_s$, and let $\Z_S$ denote the ring of $S$-integers, these are those $x \in \Q$ for which
$|x|_s \leq 1 \text{ for all } s \notin S$. In \cite{BP}, Borel and Prasad established an $S$-integer version of the Oppenheim conjecture. That is, if $Q$ is a quadratic form on $\Q_{S}^{n}$ in at least $3$ variables that is assumed to be isotropic, nondegenerate, and irrational, then $Q(\Z_{S}^{n})$ is dense in $\Q_S$. This is a slightly misleading version of events as we explain now. Firstly, Borel and Prasad proved their theorem for number fields, secondly, they were not able to prove density but only density around $0$. The reason was that at the time of their writing, the $S$-adic Raghunathan conjectures were not yet known. At present, the $S$-adic Oppenheim conjecture is known, in its density avatar, for all number fields.\\

In \cite{HLM, Han}, quantitative versions of the $S$-adic Oppenheim conjecture were established, while in \cite{GH}, an effective Oppenheim conjecture for generic inhomogeneous forms was established using a suitable analogue of Rogers's formula. These constitute $S$-adic analogues of the results of Ghosh, Kelmer and Yu \cite{GKY1, GKY2} and the proof proceeds by first establishing suitable analogues of Rogers's formulae for congruence quotients and affine lattices in the $S$-integer setting. The latter are likely to be of independent interest, see \cite{Han2} for an application to Diophantine approximation reminiscent of \cite{AGY}.

\section{Binary inhomogeneous forms}
In this section, motivated by the work of Kleinbock and Weiss (\cite{KW}) on values of binary quadratic forms, we will study the values of inhomogeneous binary quadratic forms at integer points.

It is very well known that the Oppenheim conjecture fails for binary indefinite quadratic forms. For example, consider
\[Q(x,y)=x^2-\eta^2 y^2=y^2\left(\frac{x}{y}-\eta\right)\left(\frac{x}{y}+\eta\right)\quad \text{for}\,\, x,y \in \mathbb R,\]
 where $\eta \in \mathbb R$ is badly approximable, i.e.,
 \[\inf_{x \in \mathbb Z,y \in \mathbb N}y^2\left|\frac{x}{y}-\eta\right|>0.\]
Then for all non-zero integer tuples $(x,y)$, $|Q(x,y)|>\delta$ for some $\delta >0$ and hence Oppenheim conjecture fails for the $Q$ as above. We also know that the set of all badly approximable numbers of $\mathbb R$, although null, is thick (i.e., its intersection with any nonempty open set has full Hausdorff dimension) in $\mathbb R$. From this, we get a large class of binary indefinite quadratic forms for which Oppenheim conjecture fails. In fact, the set of binary indefinite quadratic forms for which the Oppenheim conjecture fails is thick. Before going further, we set up some notation for rest of the paper. We set
 
\[G:=\SL(2,\mathbb R), \quad \Gamma:=\SL(2,\mathbb Z), \quad G_a:=\SL(2,\mathbb R) \ltimes \mathbb R^2, \quad \Gamma_a:=\SL(2,\mathbb Z) \ltimes \mathbb Z^2, \]
\[\mathcal L :=\text{the space of unimodular lattices in } \mathbb R^2 \cong \SL(2,\mathbb R)/\SL(2,\mathbb Z),\]
\[ \mathcal{L}_a:= \text{the space of affine unimodular lattices in } \mathbb R^2 \cong \SL(2,\mathbb R) \ltimes \mathbb R^2 /  \SL(2,\mathbb Z) \ltimes \mathbb Z^2,\]
\[\SL(2,\mathbb R) \ltimes \mathbb R^2 \curvearrowright \mathbb R^2 \,\, \text{by}\]
\[(g,v)\cdot y:=gy+v\,\,\,\,\text{for}\,\, (g,v) \in \SL(2,\mathbb R) \ltimes \mathbb R^2 \,\, \text{and} \,\, y \in \mathbb R^2.\]
\\
 
Let us fix the quadratic form $Q_0(y_1,y_2)=y_1y_2.$ Then given any binary indefinite quadratic form $Q$ there exists $\lambda \in \mathbb R \setminus \{0\}$ and $g \in \SL(n,\mathbb R)$ such that $Q(y)=\lambda Q_0(gy)$ for all $y \in \mathbb R^2.$ Now let $Q_{\xi}=(Q,\xi)$ be a inhomogeneous binary quadratic form. Then for any $y \in \mathbb R^2,$
\[Q_{\xi}(y)=Q(y+\xi)=\lambda Q_0(g(y+\xi))=\lambda Q_0(gy+g\xi)=\lambda Q_0((g,g\xi)\cdot y),\]
 where $(g,g\xi) \in \SL_2(\mathbb R) \ltimes \mathbb R^2.$ Hence given any inhomogeneous binary indefinite quadratic form $Q_{\xi},$ there exists $(g,v) \in \SL_2(\mathbb R) \ltimes \mathbb R^2$ and $\lambda \in \mathbb R \setminus \{0\}$  such that $Q_{\xi}(y)=\lambda Q_0((g,v)\cdot y)$ for all $y \in \mathbb R^2.$

\begin{definition}
    Let $ Q_{\xi}$ and $Q'_{\xi'}$ be two inhomogeneous binary quadratic forms. We say that $Q_{\xi}$ is equivalent to $Q'_{\xi'}$ (denoted by  $Q_{\xi} \sim Q'_{\xi'}$) if and only if there exists $(g,v) \in \SL_2(\mathbb R) \ltimes \mathbb R^2$ and $\lambda \in \mathbb R \setminus \{0\}$ such that  
    \begin{equation*}
   Q_{\xi}(y)=\lambda Q'_{\xi'}((g,v)\cdot y) 
    \end{equation*}
    for all $y \in \mathbb R^2.$
\end{definition}  
Given any inhomogeneous  binary indefinite quadratic form $Q_{\xi},$ one can easily see that $Q_{\xi} \sim Q_0$, where $Q_0(y_1,y_2)=y_1y_2$ for all $(y_1,y_2) \in \mathbb R^2.$ 
\begin{definition}
    Given an inhomogeneous binary quadratic form $Q_{\xi}$, we define the stabilizer of $Q_{\xi}$ as
    \begin{equation*}
        \SO(Q_{\xi}):=\{(g,v) \in G_a: Q_{\xi}((g,v)\cdot y)=Q_{\xi}(y) \,\, \text{for all} \,\, y \in \mathbb R^2\}.
    \end{equation*}
\end{definition}
We define
\begin{equation*}
    F_0:=\left\{b_{t}:t \in \mathbb R \right\},\,\, \text{where}\,\, b_t=\begin{bmatrix}
                 e^t& 0\\0& e^{-t}
\end{bmatrix}
\end{equation*} 
and 
\begin{equation*}
    F:= F_0 \ltimes \{0\}.
\end{equation*}
Let $H^+:=(H_0^+,0)$ and $H^-:=(H_0^-,0)$, where $H_0^+$ and $H_0^-$ are the expanding and contracting horospherical subgroups with respect to $b_1$ in $SL(2,\mathbb R),$ i.e.,
\begin{equation*}
    H_0^+:=\left\{h^+_t=\begin{bmatrix}
        1&t\\0&1 
    \end{bmatrix}:t\in \mathbb R\right\} 
\end{equation*}
and
\begin{equation*}
    H_0^-:=\left\{h^-_t=\begin{bmatrix}
        1&0\\t&0 
    \end{bmatrix}:t\in \mathbb R\right\}.
\end{equation*}

For any subgroup $K$ of a group $K'$, we denote the connected component of $K$ containing identity as $K^{\circ}.$
\begin{lemma}Suppose $Q_{\xi}=(Q,\xi)$ be an inhomogeneous binary indefinite quadratic form such that $Q(y)=\lambda Q_0(gy)$ for some $g \in \SL(2,\mathbb R)$ and $\lambda \in \mathbb R \setminus \{0\}.$ Then we have the following:
    \begin{enumerate}[label=(\roman*)]
        \item $\SO(Q_0)=\SO(Q_0)^{\circ}=F_0,$ 
        \item $\SO(Q)=\SO(Q)^{\circ}=g^{-1}F_0 g,$ and
        \item $\SO(Q_{\xi})=\SO(Q_{\xi})^{\circ}=(g,-\xi)^{-1}F(g,-\xi).$
    \end{enumerate}
\end{lemma}
\noindent The proof of this lemma is well known; for example, see \cite[Lemma 3.1]{BG}.\\

The set of all binary inhomogeneous indefinite quadratic forms considered up to scaling can be identified with $\mathcal{Q}=F\backslash G_a.$ Then the set of integer points of inhomogeneous binary indefinite quadratic forms is a set-valued function on $F\backslash G_a/\Gamma_a.$ By the duality principle of Cassels and Swinnerton-Dyer \cite{CSD}, and Dani \cite{D1}, the set of integer points of inhomogeneous binary quadratic forms can be realized either as a $F$-invariant function on $\mathcal{L}_a$ or a $\Gamma_a$-invariant function on $\mathcal{Q}$. This allows us to reframe dynamical properties of $F$-orbits of $\mathcal{L}_a$ as properties of quadratic forms. For example, an easy application of Mahler's compactness criterion shows that for a binary quadratic form $Q(y)=Q_0(gy)\in \mathcal{Q}$, the set $Q(\mathbb Z^2)$ has a gap at zero if and only if $F_0(g\mathbb Z^2)$ is bounded in $\mathcal{L}.$ From a result of Kleinbock and Margulis \cite{KM}, we know that the set of points of $\mathcal{L}$ whose $F_0$-orbit is bounded, is thick. Hence, the set of binary indefinite quadratic forms whose set of values at integer points has a gap at $0$ is thick in the space of binary indefinite quadratic forms. \\

The main theme here is that more sophisticated properties of the $F_0$ orbit of $\mathcal{L}$ yield a finer description of the set of values of a binary indefinite quadratic form at integer points. In \cite{KW}, Kleinbock and Weiss proved that given any countable set $B \subset \mathbb R$, the set of binary indefinite quadratic forms whose set of values at integer points avoids $B$ is thick in the space of all binary indefinite quadratic forms. They proved this result by proving a dynamical counterpart that the set points $\Lambda$ of $\mathcal{L}$ such that $F_0$-orbit of $\Lambda$ is bounded and the closure of $F_0$-orbit of $\Lambda$ avoids a countable union of manifolds (satisfying certain transversality conditions), is thick.\\

This section aims to explore the theme mentioned above for inhomogeneous binary indefinite quadratic forms and to generalise the main result of \cite[Theorem 1.3]{KW} in an inhomogeneous setting for the following cases:
\begin{enumerate}[label=(\roman*)]
\item Inhomogeneous binary indefinite quadratic forms where both forms and shifts are allowed to vary.
 \item Inhomogeneous binary indefinite quadratic forms with fixed shifts.
 \end{enumerate}
We now present the main results of this section. 
\begin{theorem}\label{thm:main_vary_both}
    For any countable set $B$ of $\mathbb R$ that does not contain zero, the set 
    \begin{equation*} \left \{\Lambda \in \mathcal L_a: \overline{Q_0(\Lambda)} \cap B = \emptyset \right\}
    \end{equation*}
    is thick. Dually, the set of binary inhomogeneous indefinite quadratic forms whose set of values at integer points avoids a given countable set of $\mathbb R$ not containing zero is thick in the space of all binary inhomogeneous indefinite quadratic forms.
\end{theorem}

Next, we consider the set of all affine unimodular lattices with fixed shift $\xi_0 \in \mathbb R^2$
\begin{equation*}
    \mathcal{L}_{a}^{(\xi_0)}:=\{g\mathbb Z^2 + \xi_0 : g \in \SL(2,\mathbb R)
\}\end{equation*}
Note that this set is naturally identified with the space of unimodular lattices $\mathcal L.$
\begin{theorem}\label{thm:main_vary_form}
    For any countable set $B$ of $\mathbb R$ not containing zero, the set 
    \begin{equation*}
        \left\{\Lambda \in \mathcal{L}_a^{(\xi_0)}: \overline{Q_0(\Lambda)} \cap B = \emptyset \right\}
    \end{equation*}
    is thick in $\mathcal{L}_a^{(\xi_0)} \cong \mathcal L.$ Dually, the set of binary inhomogeneous indefinite quadratic forms with fixed shift whose set of values at integer point avoids a given countable set of $\mathbb R$ not containing zero is thick in the space of all binary inhomogeneous indefinite quadratic forms with fixed shift.
\end{theorem}
 

Given a non-zero constant $s \in \mathbb R$ and an indefinite binary quadratic form $Q_{\xi},$ first we dynamically interpret the set of binary quadratic forms whose values at integer points miss $s$. We follow the approach of Kleinbock and Weiss (\cite[\S 2]{KW}).\\ 

To start with, fix $0 \neq s \in \mathbb R$. Then one can easily choose a vector $v\in \mathbb R^2$ such that $Q_0(v)=s.$ For example, we can choose 
\begin{equation} \label{eqn:v}   
v= \left \{ \begin{array}{lcl}  (\sqrt{s},\sqrt{s}) & \mbox{if} \,\, s>0 \\ (-\sqrt{-s},\sqrt{-s})&\mbox{otherwise}.
\end{array} \right.\end{equation}
 We fix this choice of $v$ for the rest of this section. Now consider the set $M_v$ of all affine unimodular lattice of $\mathbb R^2$ which contains $v$, i.e.,
 
\begin{equation}\label{equ:M_v}
M_v :=\{\Lambda\in \mathcal L_a: v \in \Lambda \}.
\end{equation}
Note that $M_v $ is a $3$-dimensional submanifold of $\mathcal L_a.$ Now we have the following proposition.
\begin{proposition}\label{prop:Q_dyn}
    Let $\Lambda \in \mathcal L_a$ be such that $F\Lambda$ is bounded in $\mathcal L_a.$ Also let $s\neq 0,$ now if $v$ is defined as earlier in \eqref{eqn:v}. Then $s \notin \overline{Q_0(\Lambda)}$ if and only if $\overline{F\Lambda} \cap M_v=\emptyset.$
\end{proposition}
\begin{proof}
First suppose that $s \notin \overline{Q_0(\Lambda)}$. We want to show that $\overline{F\Lambda} \cap M_v=\emptyset$. 

If possible let, $\Lambda_0 \in \overline{F\Lambda}$ contains $v.$ Then there exists $\{t_n\}_n$ such that $b_{t_n}\Lambda \rightarrow \Lambda_0.$ This implies there exists $\{v_n\}_n \subset \Lambda $ such that $b_{t_n}v_n \rightarrow v.$ Now $Q_0(v_n)=Q_0(b_{t_n}v_n) \rightarrow Q_0(v)=s,$ and this is contradiction to the fact that $s \notin \overline{Q_0(\Lambda)}$. Hence, we are done.

Conversely suppose that $\overline{F\Lambda} \cap M_v=\emptyset$ and if possible let $s \in \overline{Q_0(\Lambda)}$. Then $\exists$ $\{u_n\}_n \subset \Lambda $ such that $Q_0(u_n)\rightarrow s$ as $n\rightarrow \infty.$ Now for each $n \in \mathbb N,$ we choose $b_{t_n}$ such that $b_{t_n}u_n $ belongs to the line passing through $v$ and therefore $b_{t_n}u_n \rightarrow v $ as $n \rightarrow \infty.$ 

Again since $F\Lambda$ is relatively compact, $\exists$ a limit point $\Lambda_0$ of the sequence $\{b_{t_n}\Lambda\}_n \subset F\Lambda$. From above, it is clear that $v \in \Lambda_0$. Therefore $\Lambda_0 \in \overline{F\Lambda} \cap M_v,$ and this is a contradiction.

\end{proof}
We prove Theorem \ref{thm:main_vary_both} by proving its dynamical reformulation. To dynamically reformulate Theorem \ref{thm:main_vary_both}, we need the following definitions.
\begin{definition}
    Let us assume that $H$ be a connected subgroup of $G_a$ and $H \neq F.$ Also let $M $ be a submanifold of $\mathcal L_a.$ We say that $M$ is $(F,H)$-transversal at $x \in M$ if the following holds:
    \begin{enumerate}[label=(\Roman*)]
        \item\label{item:cond_F} $T_x(Fx) \not\subset T_x(M)$ for any $x \in M$
        \item\label{item:cond_F_H}$T_x(Hx)\not\subset T_x(M) \oplus T_x(Fx)$ for any $x \in M$.
    \end{enumerate}
Furthermore, we say that $M$ is $(F,H)$-transversal if $M$ is $(F,H)$-transversal at every point of $M.$
\end{definition}
It is easy to see that if $M$ is an orbit of a Lie subgroup $K$ of $G_a$, then condition \ref{item:cond_F} and \ref{item:cond_F_H} can be restated as 
\begin{equation}
    \Lie(F) \nsubset \Lie(K)
\end{equation}
and 
\begin{equation}
    \Lie(H) \nsubset \Lie(F) \oplus \Lie(K)
\end{equation}
respectively, where $\Lie(L)$ is the Lie algebra of the Lie group $L.$
\begin{lemma}\label{lem:M_v_trans}
   $M_v$ defined as in \eqref{equ:M_v} is a $3$-dimensional $(F,H^+)$- and $(F,H^-)$-transversal submanifold of $\mathcal{L}_a.$ 
\end{lemma}
 \begin{proof}
     Consider the natural embedding of the space of unimodular lattices of $\mathbb R^2$ in the space of affine unimodular lattices of $\mathbb R^2$ and denote it by $i.$ Hence $i : \mathcal L \rightarrow \mathcal{L}_a$ given by $i(\Lambda)=\Lambda$ for any $\Lambda \in \mathcal L.$ Now consider the diffeomorphism $h: \mathcal{L}_a \rightarrow \mathcal{L}_a$ given by
     \begin{equation*}
         h(\Lambda)=\Lambda +v \,\,\,\, \text{for any}\,\, \Lambda \in \mathcal{L}_a.
     \end{equation*}
Then we have $h \circ i(\Lambda)=M_v$. From this, we get that $M_v $ is a 3-dimensional embedded submanifold of $\mathcal{L}_a$, since $\mathcal L$ is a 3-dimensional embedded submanifold of $\mathcal{L}_a.$

Note that $M_v$ is an orbit of the Lie subgrup $K_v$ of $G_a$, where $K_v=\{(I_2,v)(g,0)(I_2,v)^{-1}=(g,v-gv): g \in \SL_2(\mathbb R)\}.$ Recall that if $v=(v_1,v_2)$, then both the coordinates of $v$ are non-zero, i.e., $v_1,v_2 \neq 0$.  Hence the $(F,H^+)$- and $(F,H^-)$-transversality of the submanifold $M_v$ follows from the following:
\begin{equation*}
    \Lie(K_v)= \left\langle\left\{\left(\begin{bmatrix}
        1&0\\0&-1 
    \end{bmatrix},\begin{pmatrix}
        -v_1 \\ v_2
    \end{pmatrix}\right), \left(\begin{bmatrix}
        0&1\\0&0 
    \end{bmatrix}, \begin{pmatrix}
         -v_2 \\ 0
    \end{pmatrix}\right), \left(\begin{bmatrix}
        0&0\\1&0 
    \end{bmatrix},\begin{pmatrix}
        0 \\ -v_1
    \end{pmatrix}\right) \right\}\right\rangle,
\end{equation*}
\begin{equation*}
    \Lie(F)=\left\langle\left\{\left(\begin{bmatrix}
        1&0\\0&-1 
    \end{bmatrix}, \begin{pmatrix}
        0\\0
    \end{pmatrix}\right) \right\}\right\rangle, \Lie(H^+)=\left\langle\left\{\left(\begin{bmatrix}
        0&1\\0&0 
    \end{bmatrix},\begin{pmatrix}
        0\\0
    \end{pmatrix}\right) \right\}\right\rangle,
\end{equation*}
and
\begin{equation*}
    \Lie(H^-)=\left\langle\left\{\left(\begin{bmatrix}
        0&0\\1&0 
    \end{bmatrix},\begin{pmatrix}
        0\\0
    \end{pmatrix}\right) \right\}\right\rangle.
\end{equation*}
Here given a subset $A$ of the vector space $M_{2\times 2}(\mathbb R)$ of $2\times 2 $ matrices over $\mathbb R,$ $\langle A \rangle$ denotes the subspace of $M_{2\times 2}(\mathbb R)$ generated by A. 
 \end{proof}

Now the following are the dynamical reformulations of Theorem \ref{thm:main_vary_both} and Theorem \ref{thm:main_vary_form} respectively.

\begin{theorem}\label{thm:main_vary_both_dyn}
    Suppose that $M$ be a countable union of submainfolds (of $\mathcal L_a$), which are both $(F,H^+)$- and $(F,H^-)$-transversal. Then the following set
    \begin{equation}\label{equ:main_dyn_vary_both}
        \{ \Lambda \in \mathcal L_a:F\Lambda \,\, \text{is bounded in}\,\, \mathcal L_a \,\, \text{and}\,\, \,\,\overline{F\Lambda} \cap M =\emptyset \}
    \end{equation}
    is thick.
\end{theorem}
 
\begin{theorem}\label{thm:main_fix_shift_dyn}
  Suppose that $M$ be a countable union of submainfolds (of $\mathcal L_a$), which are both $(F,H^+)$- and $(F,H^-)$-transversal. Then the following set
    \begin{equation}\label{equ:main_dyn_fix_shift}
        \{ \Lambda \in \mathcal{L}_a^{(\xi_0)}:F\Lambda \,\, \text{is bounded in}\,\, \mathcal L_a \,\, \text{and}\,\, \,\,\overline{F\Lambda} \cap M =\emptyset \}
    \end{equation}
    is thick in   $\mathcal{L}_a^{(\xi_0)} \cong \mathcal L.$
\end{theorem}

Since the space of affine unimodular lattices is a torus bundle over the space of unimodular lattices, Theorem \ref{thm:main_fix_shift_dyn} together with Marstrand’s slicing theorem (\cite{Mars} or \cite[Theorem 5.8]{F}) implies Theorem \ref{thm:main_vary_both_dyn}. We again note that both Theorem \ref{thm:main_vary_both_dyn} and Theorem \ref{thm:main_fix_shift_dyn} follow from more general result \cite[Theorem 2.8]{AGK} by An, Guan, and Kleinbock (apply Theorem 2.8 of \cite{AGK} for $G=G_a,\,\, \Gamma=\Gamma_a,\,\,X=\mathcal{L}_a, \,\,F=F_{0} \ltimes \{0\}, \,\,Z=M,\,\, \text{and}\,\, H=G_a$ to get Theorem \ref{thm:main_vary_both_dyn}. For deducing Theorem \ref{thm:main_fix_shift_dyn} just change the $H$ to $H=(I,\xi_0)(\SL(2,\mathbb R) \ltimes \{0\})(I,\xi_0)^{-1}$ in the previous case). Hence, we don't claim the originality of these two dynamical theorems. Now, we will see how Theorem  \ref{thm:main_vary_both} and \ref{thm:main_vary_form}  follow from Theorem \ref{thm:main_vary_both_dyn} and \ref{thm:main_fix_shift_dyn}     respectively. Let's work out the reduction for Theorem \ref{thm:main_vary_both}. Proof of the other reduction is analogous.
\begin{proof}[Proof of Theorem \ref{thm:main_vary_both} modulo Theorem \ref{thm:main_vary_both_dyn}] Given any $b \in B,$ consider $v=v(b)$ as in \eqref{eqn:v} and take $M = \bigcup_{b \in B}M_{v(b)}.$ Then by Lemma \ref{lem:M_v_trans}, $M$ is a countable union of $(F,H^+)$- and $(F,H^-)$-transversal submanifolds of $\mathcal{L}_a$. Hence, Theorem \ref{thm:main_vary_both_dyn} implies that the set \eqref{equ:main_dyn_vary_both} is thick. Now Proposition \ref{prop:Q_dyn} implies that any $\Lambda $ belong to \eqref{equ:main_dyn_vary_both} satisfies $\overline{Q_0(\Lambda)} \cap B =\emptyset.$ Therefore, Theorem \ref{thm:main_vary_both} follows. 

\end{proof}
We finish this section with the remark that all these theorems are proved using the technique of Schmidt games. To illustrate the method, we present the proof of Theorem \ref{thm:main_vary_both_dyn}. 
The proof follows the same line of argument as that of An, Guan, and Kleinbock (\cite{AGK}), 
and Kleinbock and Weiss~\cite{KW}. However, we hope that this exposition will be of interest to the reader.

\section{Schmidt Games}
We use some variation of the Schmidt game to prove our results. To begin, recall Schmidt's $(\alpha,\beta)$-game from \cite{S}. The game is played between two players, Alice and Bob, on a complete metric space 
$(Y,d)$. The game involves a target set $S \subseteq Y$ and two parameters $\alpha, \beta \in (0,1).$ Bob begins the $(\alpha, \beta)$-game by selecting a point $ y_1 \in Y $ and a radius $ r_1 > 0 $, defining the closed ball $ B_1 = \overline{B}(y_1, r_1) $, where:
\[
\overline{B}(y,r) = \{z \in Y : d(z, y) \leq r\}.
\]

\noindent The game proceeds with Alice and Bob alternately choosing points $ y_i' $ and $ y_{i+1} $, subject to the constraints:
\[
d(y_i, y_i') \leq (1 - \alpha)r_i, \quad d(y_{i+1}', y_i) \leq (1 - \beta)r_i',
\]
where:
\[
r_i' = \alpha r_i, \quad r_{i+1} = \beta r_i'.
\]

\noindent This process ensures that the closed balls
\[
A_i = \overline{B}(y_i', r_i'), \quad B_{i+1} = \overline{B}(y_{i+1}, r_{i+1}),
\]
are nested, i.e.,
\[
B_1 \supset A_1 \supset B_2 \supset \cdots.
\]

The target set $ S \subseteq Y $ is said to be \emph{$\alpha$-winning} if for any $ \beta > 0 $, Alice has a strategy in the $(\alpha, \beta)$-game that guarantees the unique point of intersection
\[
\bigcap_{i=1}^\infty B_i = \bigcap_{i=1}^\infty A_i,
\]
belongs to $ S $, regardless of Bob's choices. The set $ S $ is called \emph{winning} if it is $\alpha$-winning for some $\alpha \in (0, 1)$.\\

Inspired by ideas from McMullen (\cite{M}), in \cite{BFKRW}, the \emph{absolute hyperplane game} was introduced as a modification played on $\mathbb{R}^n $. Let $S \subseteq \mathbb{R}^n$ be the target set, and let $\beta \in \left( 0, \frac{1}{3} \right) $.

\noindent The game proceeds as follows:  

\begin{itemize}
    \item  Bob begins by selecting a closed ball $B_1$ of radius $r_1$.  
    \item  Alice and Bob then alternate turns.  
    \item On each turn, Bob chooses a closed ball $B_{i+1}$ of radius $r_{i+1}$, ensuring:  
     \[
    r_{i+1} \geq \beta r_i.
    \]
    \item Alice chooses sets $A_i$, which are $r_i'$-neighborhoods of affine hyperplanes, where:  
    \[
    r_i' \leq \beta r_i.
    \]
    \item Additionally, Bob must ensure that his chosen ball satisfies the condition:  
    \[
B_{i+1} \subset B_i \setminus A_i.
\]
\end{itemize}

\noindent A set $S\subset \mathbb{R}^d$ is called \emph{$\beta$-HAW} (HAW stands for \emph{hyperplane absolute winning}) if Alice has a strategy that ensures:  
\[
\bigcap_{i=1}^\infty B_i \cap S \neq \emptyset,
\]  
regardless of Bob's choices. The set $S$ is called \emph{HAW} if it is $\beta$-HAW for all $0 < \beta < \frac{1}{3}$. It is straightforward to observe that $\beta$-HAW implies $\beta'$-HAW whenever $\beta \leq \beta' < \frac{1}{3}$. Thus, a set $S$ is HAW if and only if it is $\beta$-HAW for arbitrarily small positive values of $\beta$.

 The following proposition outlines key properties of winning and HAW subsets of $\mathbb{R}^n $:  

\begin{proposition}\label{prop:winning_HAW_properties}  
\begin{enumerate}[label=(\roman*)] 
    \item Winning sets are thick.  
    \item The HAW property implies the winning property.  
    \item The countable intersection of $\alpha$-winning (respectively, HAW) sets is again $\alpha$-winning (respectively, HAW).  
    \item The image of a HAW set under a $ C^1 $-diffeomorphism $ \mathbb{R}^d \to \mathbb{R}^d $ is HAW. 
    
\end{enumerate}  
\end{proposition}  
\noindent For detailed proofs, see \cite{S, M, BFKRW}. 

For our purpose, we want to play variants of these games on differentiable manifolds. Following \cite{KW}, we recall the absolute hyperplane game on $C^1$ manifolds.
\begin{itemize}
    \item We start with defining the absolute hyperplane game on an open subset $O $ of $\mathbb R^n.$ This is defined in exactly the same manner except we demand that the first choice $B_1$ of Bob should contained inside $O.$
    \item For a $n$ dimensional $C^1$ manifold $Y$ with a system $\{(U_{\alpha},\phi_{\alpha})\}$ of coordinate charts giving a differentiable structure to $Y$, a subset $S$ of  $Y$ is called HAW on $Y$ if $\phi_{\alpha}(U_{\alpha} \cap S)$ is HAW
on $\phi_{\alpha}(U_{\alpha})$ for each $\alpha$. For more details on the absolute hyperplane game on $C^1$ manifolds, see \cite[\S 3]{KW}.
\end{itemize}
As mentioned earlier, we record a few properties of the absolute hyperplane game on $C^1$ manifolds in the following proposition (for proofs see \cite{KW,AGK}).
\begin{proposition}\label{prop:HAW_manifold_pro}  
\begin{enumerate}[label=(\roman*)] 
    \item HAW subsets of a $C^1$ manifold are thick.
    \item The countable intersection of HAW subsets of a $C^1$ manifold is also HAW.
    \item For a $C^1$ manifold $Y$, the image of a HAW subset of $Y$ under a $C^1$ diffeomorphism $Y \rightarrow Y$ is HAW.
  \item Let $f:Y_1 \rightarrow Y_2$ be a surjective $C^1$ submersion of $C^1$ manifolds. If $S \subset Y_1$ is a HAW set, then its image $f(S) \subset Y_2$ is also HAW.
\end{enumerate}
\end{proposition}

\subsection{Hyperplane percentage game} For our purpose, we need the following variant of Schmidt games. Fix a parameter $\beta>0.$ The hyperplane percentage game on $\mathbb R^n$ with target set $S$ is played as follows:
\begin{itemize}
    \item As earlier, Bob begins by choosing a $B_1 \subseteq \mathbb R^n.$
    \item After Bob chooses a ball $B_i$
  with radius $r_i$ , Alice selects finitely many affine hyperplanes $\mathcal H_{i,j}$
  and corresponding values $\varepsilon_{i,j}$
 , where $j=1,\dots,N_i$. The parameters satisfy $0<\varepsilon_{i,j} \leq \beta r_i$ and $N_i$ is any positive integer chosen by Alice.
 \item Bob then selects a ball $B_{i+1}  \subset B_i$ with radius $r_{i+1} \geq \beta r_i$, ensuring that
\begin{equation*}
    B_{i+1} \cap \mathcal{H}_{i,j}^{(\varepsilon_{i,j})}=\emptyset
\end{equation*}
for at least $\frac{N_i}{2}$ values of $j.$
\item This process generates a nested sequence of closed balls $B_1 \subset B_2 \subset \dots$, and Alice is declared the winner if and only if
\begin{equation*}
    \bigcap_{i=1}^{\infty} B_i \cap S \neq \emptyset
\end{equation*}
If Alice has a winning strategy regardless of Bob's moves, the set $S$ is called $\beta$-hyperplane percentage winning (or
$\beta$-HPW).
\item For sufficiently large values of $\beta$, Alice may force Bob into a position where he cannot make a valid move after finitely many turns. However, one can show that (see \cite[Lemma 2]{Mo} or \cite[\S 2]{BFK}) shows that Bob always has valid moves when 
$\beta$ is smaller than a certain constant $\beta_0(n)$. For example $\beta_0(1)=1/5.$
\item If $S$ is $\beta$-HPW for all $0 <\beta < \beta_0(n)$, it is said to be hyperplane percentage winning (HPW). Moreover, it is clear that $\beta$-HPW implies $\beta'$-HPW whenever 
$\beta \leq \beta'$, and thus HPW is equivalent to $\beta$-HPW for arbitrarily small values of 
$\beta$.
\end{itemize}
 The hyperplane percentage game is inherently more advantageous for Alice compared to the hyperplane absolute game. Consequently, every hyperplane absolute winning (HAW) set is also a hyperplane percentage winning (HPW) set. Remarkably, the reverse implication holds as well (for details, see \cite[Lemma 2.1]{BFS}).
 \begin{lemma}
     Given any $0 < \beta < 1/3,$ $\exists$ $\beta' \in (0,\beta_0(n)) $ such that any set which is $\beta'$-HPW is $\beta$-HAW. In other words, the HAW and HPW are equivalent properties.
 \end{lemma}
We end this section with the following remark.
\begin{remark}
    For proving a set $S$ to be HPW, without loss of generality, we can always assume that the radii $r_i \rightarrow 0$ as $i \rightarrow \infty.$ Indeed, if Alice has a winning strategy under the assumption that $r_i \rightarrow 0$, then the set $S$ must be dense. In that case, she automatically wins in any play where $r_i \not\to 0$, since the balls cannot avoid a dense set indefinitely. Furthermore, by allowing Alice to make a few dummy moves at the start (e.g., not removing anything substantial), and by relabelling the balls $B_i$, we may also assume that the initial radius $r_0$ is smaller than any given positive constant.
\end{remark}

\section{Reduction to discrete time action}
In this section, we reduce Theorem \ref{thm:main_vary_both_dyn} to discrete time action. To start with, we have the following easy reduction. Define $F^+:=\{(b_t,0)\in F: t \geq 0\}.$ Then Theorem \ref{thm:main_vary_both_dyn} can be reduced to the following:

\begin{theorem}\label{thm:main_vary_both_dyn_redu1}
\begin{enumerate}[label=(\roman*)]
\item The set
\begin{equation}\label{equ:1}
 \left\{\Lambda \in \mathcal{L}_a:F^+ \Lambda \,\, \text{is bounded in} \,\, \mathcal{L}_a \right\}
\end{equation}
is HAW.
\item Suppose that $M$ is a compact $(F,H^+)$-transversal submanifold of $\mathcal L_a.$ Then the set 
\begin{equation}\label{equ:2}
 \left\{\Lambda \in \mathcal{L}_a:\overline{F^+ \Lambda} \cap M=\emptyset \right\}
\end{equation}
is HAW.
 \end{enumerate}   
\end{theorem}
\begin{proof}[Proof of Theorem \ref{thm:main_vary_both_dyn} modulo Theorem \ref{thm:main_vary_both_dyn_redu1}]
  First, it is easy to observe that similar statements are also true for the semigroup $F^- :=\{(b_t,0) \in F:t \leq 0\}$ instead of $F^+$ and with the roles of $H^+$ and $H^-$ exchanged. That is, the sets
  \begin{equation}\label{equ:3}
      \left\{\Lambda \in \mathcal{L}_a:F^- \Lambda \,\, \text{is bounded in} \,\, \mathcal{L}_a \right\}
  \end{equation}
  and
  \begin{equation}\label{equ:4}
    \left\{\Lambda \in \mathcal{L}_a:\overline{F^- \Lambda} \cap M=\emptyset \right\},
  \end{equation}
  where $M$ is a compact $(F,H^-)$-transversal submanifold of $\mathcal L_a$, are HAW. Now the set
  \begin{equation}\label{equ:5}
      \{ \Lambda \in \mathcal{L}_a:F\Lambda \,\, \text{is bounded in}\,\, \mathcal L_a \,\, \text{and}\,\, \,\,\overline{F\Lambda} \cap M =\emptyset \}
  \end{equation}
  is the intersection of sets described \eqref{equ:1}-\eqref{equ:4}. Therefore, by Proposition \ref{prop:HAW_manifold_pro}, \eqref{equ:5} is hyperplane absolute winning in $\mathcal{L}_a$, under the hypothesis that $M$ is a both $(F,H^+)$ and $(F,H^-)$-transversal submanifold of $\mathcal{L}_a.$

   Therefore, by Proposition \ref{prop:HAW_manifold_pro}, \eqref{equ:5} is HAW (hyperplane absolute winning) whenever $M \subset \mathcal{L}_a$ is compact and it is both $(F,H^+)$- and $(F,H^-)$-transversal. In Theorem \ref{thm:main_vary_both_dyn}, the set $M$ is given as a countable union of submanifolds. By refining if necessary, each such submanifold can be replaced by a countable union of compact manifolds with boundary. Hence, the set in equation \eqref{equ:main_dyn_fix_shift} becomes a countable intersection of sets of the form \eqref{equ:5}. Now  Theorem \ref{thm:main_vary_both_dyn} follows from Proposition \ref{prop:winning_HAW_properties}.
\end{proof}
\begin{proof}[Proof of Theorem \ref{thm:main_vary_both_dyn_redu1}(i)] Note that $F^+ \Lambda$ is bounded in $\mathcal{L}_a$ if and only if $\pi(F^+ \Lambda)$ is bounded in $\mathcal L,$ where $\pi:\mathcal{L}_a \rightarrow \mathcal{L}$ is the natrural projection map given by $\pi(g\mathbb Z^2+\xi)=g\mathbb Z^2$ for $g\mathbb Z^2 +\xi \in \mathcal{L}_a.$ 

We already know that the set $\left\{\Lambda \in \mathcal{L}:F^+ \Lambda \,\, \text{is bounded in} \,\, \mathcal{L}\right\}$   is HAW in $\mathcal L.$  Hence the set $\left\{\Lambda \in \mathcal{L}_a :F^+ \Lambda \,\, \text{is bounded in} \,\, \mathcal{L}_a \right\}$ is HAW in $ \mathcal{L}_a.$
\end{proof}
In view of the above, we only need to prove Theorem \ref{thm:main_vary_both_dyn_redu1}$(ii)$. To do that, we further reduce it to discrete time actions. We closely follow \cite[\S 4]{K} and \cite[\S 4]{KW}. For a parameter $\sigma>0,$ we define
\begin{equation}
    F_{\sigma}^{+}:=\{(b_{n\sigma},0):n \in \mathbb N \cup \{0\}\}.
\end{equation}
$F_{\sigma}^{+}$ is the cyclic subsemigroup of $F$ generated by $(b_{\sigma},0).$ We want to show that one can replace the continuous semigroup $F^+$ with $F^{+}_{\sigma}$ in Theorem \ref{thm:main_vary_both_dyn_redu1}$(ii)$. Before proceeding further, we recall the following definition:
\begin{definition}
    Suppose that $H$ is a subgroup of $G_a$ and $M$ is a smooth submanifold of $\mathcal{L}_a.$ We say that $M$ is $H$-transversal at $y \in M$ if $T_{y}(Hy)$ is not contained in $T_y M.$ Futhermore $M$ is said to be $H$-transversal if $M$ is $H$-transversal at every point of M.
\end{definition}
Also given $M \subseteq \mathcal{L}_a$ and $t_1,t_2 \in \mathbb R,$ we define the following:
\begin{equation}
  M_{[t_1,t_2]}:=\bigcup_{t_1 \leq t \leq t_2} b_tM. 
\end{equation}
Now we have the following lemma, whose proof is analogous to \cite[Lemma 4.1]{KW}. We give proof of it for the sake of completeness.
\begin{lemma}\label{lem:M_sigma_manifold}
 Let $M$ be a compact $C^1$ submanifold of $\mathcal{L}_a$ and $H$ be a connected subgroup of $G_a.$ Also suppose that $M$ is $(F,H)$-transversal. Then we have the following:
 \begin{enumerate}[label=(\roman*)]
\item There exists a $\delta=\delta(M)>0$ such that $M_{[-\delta,\delta]}$ is a $C^1$ manifold.
\item  There exists a $\tau=\tau(M) \leq \delta(M)$ such that $T_y(Hy)$ is not contained in $T_y(M_{[0,\tau]})$ for any $y \in M_{[0,\tau]}$, in other words $M_{[0,\tau]}$ is $H$-transversal.
 \end{enumerate}  
\end{lemma}
\begin{proof}
    Suppose that $\dim(M)=k.$ Since $M$ is compact $C^1$ submanifold of $\mathcal{L}_a,$ by covering $M$ with a finite number of suitable coordinate charts of $\mathcal{L}_a$, we can, without loss of generality, assume that $M$ is of the form $f(W)$ for some bounded open set $W \subset \mathbb{R}^k$, where $f$ is a $C^1$ nonsingular embedding defined on an open set $W' \subset \mathbb{R}^k$ that strictly contains $W$.  Now, define $\tilde{f}: W' \times \mathbb{R} \to \mathcal{L}_a$ by setting  
\[
\tilde{f}(w, t) = b_t(f(w)).
\]  
Since $M$ is $F$-transversal, it follows that $\tilde{f}$ is nonsingular at $t = 0$ for all $w \in W$. Consequently, $\tilde{f}$ serves as a nonsingular embedding of $W'' \times [-\sigma, \sigma]$ into $\mathcal{L}_a$ for some $\sigma > 0$ and an open set $W''$ strictly containing $W$, thus establishing (i). 

It is evident that the tangent space to $ M_{[-\sigma, \sigma]} $ at any point $ y \in M $ is given by $ T_y M \oplus T_y (Fy) $. Consequently, condition $ (H, F) $ ensures that $ M_{[-\sigma, \sigma]} $ is $ H $-transversal at every point of $ M $. Since $ H $-transversality is an open condition, it must also hold at any point of $ M_{[-\sigma, \sigma]} $ that is sufficiently close to $ M $. By compactness, one can choose a positive $ \tau \leq \sigma $ such that $ M_{[0, \tau]} $ remains $ H $-transversal.
\end{proof}

Here is an alternative way to express $H$-transversality: fix a Riemannian metric $\text{dist}$ on the tangent bundle of $\mathcal{L}_a $, and for a $ C^1 $ submanifold $ M \subset \mathcal{L}_a $, consider the function $ \theta_H: M \to \mathbb{R} $ defined by  
\[
\theta_H(y):=\sup_{w \in T_y(Hy), \|w\|=1} \text{dist}(w, T_y M).
\]
Clearly, $ \theta_H(y) \neq 0 $ if and only if $M$ is $H$-transversal at $y$, and $\theta_H$ is continuous in $y \in M $. This leads to the following result:
\begin{lemma}
     A compact $ C^1 $ submanifold $ M$ of $\mathcal{L}_a $ is $ H$-transversal if and only if there exists $ \eta = \eta(M) > 0 $ such that $ \theta_H(y) \geq \eta $ for all $ y \in M $.
\end{lemma}

A right-invariant metric on $G_a$ induces a well-defined Riemannian metric on $\mathcal{L}_a$, and we now slightly modify our notation, using $d$ to denote the resulting path metric on $\mathcal{L}_a$. 

Now let $M$ be as in Theorem \ref{thm:main_vary_both_dyn_redu1}, i.e., compact and $(H^+, F)$-transversal, and choose a positive $ \tau \leq \tau(M) $ as in Lemma \ref{lem:M_sigma_manifold}, satisfying the additional condition:
\begin{equation} \label{equ:distortion}
\Lambda, \Lambda' \in \mathcal{L}_a, \quad 0 \leq t \leq \tau \quad \Rightarrow \quad d(b_t \Lambda, b_t \Lambda') \leq 2 \, d(\Lambda, \Lambda').
\end{equation}
This is possible because $ b_t $, for $ |t| \leq \tau $, is bounded and hence there exists a uniform bound on the amount by which it distorts the Riemannian metric.

Now, suppose that for some $ \Lambda \in \mathcal{L}_a $ and $ \varepsilon > 0 $, there exists $ t \geq 0 $ and $ \Lambda_1 \in M $ such that the distance between $ b_t \Lambda $ and $ \Lambda_1 $ is less than $ \varepsilon $. Choose $ 0 \leq t_1 < \tau $ such that $ t + t_1 = n\tau $ for some $ n \in \mathbb{N} $. Then, from \eqref{equ:distortion}, it follows that
\[
d(b_{n\tau} \Lambda, b_{t_1} \Lambda_1) < 2\varepsilon.
\]
This elementary argument establishes the claim:
\[
F^+_{\tau} \Lambda \cap M_{[0,\tau]} = \emptyset \quad \Rightarrow \quad F^+ \Lambda \cap M = \emptyset.
\]
Hence, Theorem \ref{thm:main_vary_both_dyn_redu1} reduces to the following:
\begin{theorem}
    Suppose that $M$ be a compact $H^+$-transversal submanifold of $\mathcal{L}_a.$ Then the set
    \begin{equation}\label{equ:mainset_xi}
         \left\{\Lambda \in \mathcal{L}_a :F^+_{\tau} \Lambda \cap M=\emptyset \right\}
    \end{equation}
    is HAW.
\end{theorem}
\begin{proof}
Let $\beta>0.$ We want to show that the set \eqref{equ:mainset_xi} is $\beta$-HPW. Since we want to play the game on manifolds, before going further, let us fix an atlas of coordinate charts for $ \mathcal{L}_a.$ We denote the Lie algebra of $G_a$ by 
\begin{equation*}
    \mathfrak{g}_a:= \Lie(G_a)= \mathfrak{sl_2}(\mathbb R) \ltimes \mathbb R^2.
\end{equation*}
Given any $\Lambda \in \mathcal{L}_a,$ define the following map $\Phi_{\Lambda}: \mathfrak{g}_a \rightarrow \mathcal{L}_a $ defined by
\begin{equation}
    \Phi_{\Lambda}(\mathbf{x})=\exp(\mathbf x)\Lambda  \quad \text{for}\,\, \mathbf x \in \mathfrak{g}_a.
\end{equation}
Then for any $\Lambda \in \mathcal{L}_a,$  there exists a neighbourhood $W_{\Lambda}$ of $ 0 \in \mathfrak{g}_a$ such that $\Phi_{\Lambda}|_{W_{\Lambda}}$ is one-to-one. We denote 
\begin{equation}
V_{\Lambda}:= \Phi_{\Lambda}(W_{\Lambda}) \,\, \text{and}\,\, \Psi_{\Lambda}=\Phi_{\Lambda}^{-1}|_{V_{\Lambda}}
\end{equation}
We identify $\mathfrak{g}_a$ with $\mathbb R^5$ and since given any $\Lambda \in \mathcal{L}_a$, we have $\Lambda \in V_{\Lambda}$, therefore $\mathcal{L}_{a}=\bigcup_{\Lambda \in \mathcal{L}_a} V_{\Lambda}$. Hence we use the collection $\{(V_{\Lambda},\Psi_{\Lambda}): \Lambda \in \mathcal{L}_a \}$ as an atlas of coordinate charts for $\mathcal{L}_{a}$.\\ 

Now it is enough to show that given any $\Lambda \in \mathcal{L}_a,$ the following set
\begin{equation}\label{equ:redu_traget_set}
\Psi_{\Lambda}(\{\Lambda' \in V_{\Lambda}:F_{\tau}^{+} \Lambda' \cap M = \emptyset\})=\{\mathbf x \in W_{\Lambda}:F_{\tau}^{+}(\exp(\mathbf x)\Lambda) \cap M=\emptyset\}
\end{equation}
is HPW on $W_{\Lambda}.$ \\

Without loss of generality, we assume that 
\begin{equation}\label{equ:beta_etau}
    \beta < e^{-2\tau}.
\end{equation}
Choose $n$ sufficiently large  so that 
\begin{equation}\label{equ:beta_e}
    \beta^{-n} < e^{2\tau(2^n-1)}.
\end{equation}
Pick an arbitrary $\Lambda  \in \mathcal{L}_{a},$ and  suppose that Bob chooses a ball $B_0 \subset W_{\Lambda}$ of radius $r_0.$ 
We want to show that Alice can play the $\beta$-hyperplane percentage game in a fashion so that 
\begin{equation}
    \bigcap_{i}B_i \in \{\mathbf{x} \in W_{\Lambda}:F_{\tau}^{+}(\exp(\mathbf{x})\Lambda ) \cap M=\emptyset\}.
\end{equation}
As mentioned earlier, it is enough to prove that Alice has a winning strategy when the radii of the balls go to zero. Therefore, the game begins with Alice making dummy moves until Bob’s ball first reaches a radius $r_0 \leq \delta.$ Re-indexing if necessary, we denote this ball as $B_1$ with radius $r_0$. Additionally, we observe that $r_0 \geq \beta \delta.$

To define Alice’s strategy, we will partition the game into stages.  We define the 
$j$th stage of the game as the set of indices 
$i \geq 0$ satisfying
\begin{equation}\label{equ:r_i}
    \beta^{n(j+1)}r_0 < r_i \leq \beta^{nj}r_0.
\end{equation}
Then each stage consists of finitely many turns and contains at least $n$ indices. Let us denote by $i_j$ the smallest index in the $j$-th stage, namely, the index at which Bob initiates the $j$-th stage by choosing the ball $B_{i_j}$ in $\mathfrak{g}_a$. In particular, the first stage begins at index $i_0=0$. From the definition of the game rule, it follows that
\begin{equation}\label{equ:r_i_j}
  \beta^{nj+1}r_0 <  r_{i_j} \leq \beta^{nj}r_0.
\end{equation}

Let us consider the following set, which we call the $j$-th window of the game
\begin{equation}\label{equ:beta_e_beta}
    \mathcal{N}_j=\{k \in \mathbb N \cup \{0\}: \beta^{-n(j-1)} \leq e^{2k\tau} < \beta^{-jn}\}.
\end{equation}
From \eqref{equ:beta_etau}, it follows that $e^{2k\tau} \geq \beta^j$ for any $j \geq 0.$ Hence, we have
\begin{equation}\label{equ:N_j}
    \bigcup_{j \in \mathbb N \cup \{0\}} \mathcal{N}_j=\mathbb N \cup \{0\}.
\end{equation}
Note that, given any $k_1,k_2 \in \mathcal{N}_j$ with $k_1 < k_2,$ we have
\begin{equation}
  e^{2(k_2-k_1) \tau}= \frac{e^{2k_2 \tau}}{e^{2k_1 \tau}}  < \frac{{\beta}^{-jn}}{{\beta}^{-n(j-1)}}=\beta^{-n} < e^{2 \tau (2^n-1)},
\end{equation}
by using \eqref{equ:beta_e} and this implies $k_2 -k_1 < 2^n-1.$ Therefore 
\begin{equation}\label{equ:N_j_count}
    \# \mathcal{N}_j < 2^n.
\end{equation}
Using \eqref{equ:r_i_j} and \eqref{equ:beta_e_beta}, for any $k \in \mathcal{N}_j,$ we get that
\[r_{i_j}e^{2k\tau}\in [\beta^{nj+1}r_0 \cdot \beta^{-n(j-1)},\beta^{nj} r_0 \cdot \beta^{-nj}]=[\beta^{n+1}r_0,r_0].\]
This implies 
\[\frac{\beta^{n+1}r_0}{r_{i_j}} \leq e^{2k\tau} \leq \frac{r_0}{r_{i_j}}. \]

Therefore, by invoking \cite[Lemma 4.6]{AGK}, for $G=G_a,$ $\Gamma=\Gamma_a, X=\mathcal{L}_a, H=G_a, Z=M$ and $f(x)=b_{\tau}x$ for $x \in \mathcal{L}_a$, we get a neighbourhood $\Omega=\Omega(\beta^{n+1},r_0)$ of $M$ such that for any $k \in \mathcal{N}_j$, there exists an affine hyperplane $L(B_{i_j},k)$ in $\mathfrak{g}_a$ such that
\begin{equation}\label{equ:Psi_g_tau}
    \Psi_{\Lambda} (b_{\tau}^{-k}(\Omega)) \cap B_{i_j} \subset L(B_{i_j},k)^{(\beta^{n+1}r_{i_j})},
\end{equation}
where $L^{(\varepsilon)}$ denotes the $\varepsilon$-neighbourhood of $L$ in $\mathfrak{g}_a.$
Let the $i_j$th move of Alice be the hyperplane neighbourhoods
\begin{equation}\label{equ:hyp_neig}
    \left\{L(B_{i_j},k)^{(\beta^{n+1}r_{i_j})}:k \in \mathcal{N}_j\right\}.
\end{equation}
For any index $i$ in the $j$-th stage, suppose Bob has chosen the ball $B_i$. Then Alice choose those neighborhoods from \eqref{equ:hyp_neig} which intersect $B_i$. 
By \eqref{equ:r_i_j} and \eqref{equ:r_i}, we have
\[
\beta^{n+1} r_{i_j} \leq \beta^{n+1}\cdot \beta^{nj} r_0 
= \beta \cdot \beta^{ n(j+1)} r_0 < \beta r_i.
\] 
Thus, Alice’s choice is consistent with the rules of the game. Next, we verify that this strategy guarantees Alice’s success. 
From the rules, if $i$ lies in the $j$-th stage, then using \eqref{equ:N_j_count}, we get
 \begin{equation}\label{equ:N_j_sub_count}
     \#\{k \in \mathcal N_j:B_{i+1} \cap {L(B_{i_j}, k)}^{(\beta^{n+1} r_{i_j})} \neq \emptyset \}\leq \frac{\# \mathcal{N}_j}{2^{i+1-i_j}} < 2^{n-(i+1-i_j)}
 \end{equation}
On the other hand, each stage has at least $n$ indices, so in particular, the index $i_j + n - 1$ belongs to the $k$-th stage. 
Substituting $i = i_j + n - 1$ into  \eqref{equ:N_j_sub_count}, we find that
\[
\# \{k \in N_j : B_{i_j+n} \cap L(B_{i_j}, k)^{(\beta^{n+1} r_{i_j})} \neq \emptyset \} < 1.
\]
This shows that
\[
B_{i_j+n} \cap L(B_{i_j}, k)^{(\beta^{n+1} r_{i_j})} = \emptyset
\quad \text{for all } k \in N_j.
\]
Combining this with \eqref{equ:Psi_g_tau}, we deduce that
\[
B_{i_j+n} \cap \Psi_{\Lambda}\big(b_{\tau}^{-k}(\Omega)\big) = \emptyset
\quad \text{for every } k \in N_j.
\]
Consequently, for each $k \in N_j$, the unique limit point
\[
\mathbf{x}_\infty \in \bigcap_{i=0}^\infty B_i
\]
is not contained in $\Psi_{\Lambda}\big(b_{\tau}^{-k}(\Omega)\big)$. Equivalently,
\[
b_{\tau}^k(\exp(\mathbf{x}_\infty)\Lambda) \notin \Omega.
\]
By \eqref{equ:N_j}, it follows that $\mathbf{x}_\infty$ lies in the target set \eqref{equ:redu_traget_set}. Therefore, Alice’s strategy ensures a win.
\end{proof}

\end{document}